\documentclass[a4paper,11pt]{amsart}
 \usepackage{float}
\usepackage{amssymb}
\usepackage{mathrsfs}
\usepackage{amsmath,amssymb,amsthm,latexsym,amscd,mathrsfs}
\usepackage{indentfirst}
\usepackage{stmaryrd}
\usepackage{graphicx}
\usepackage{extarrows}
\usepackage{multirow}
\usepackage{latexsym}
\usepackage{amsfonts}
\usepackage{color}
\usepackage{pictexwd,dcpic}
\usepackage{graphicx}
\usepackage{psfrag}
\usepackage{hyperref}
\usepackage{comment}
\usepackage{enumerate}

\numberwithin{equation}{section} \theoremstyle{plain}
\newtheorem{thm}{Theorem}[section]

\newtheorem{lem}[thm]{Lemma}

 \makeatletter

\def\<{\langle}
\def\>{\rangle}
\def\({\left(}
\def\){\right)}
\def\[{\left[}
\def\]{\right]}

\makeatother

\title{Index estimates for free boundary $f$-minimal hypersurfaces}

\author[N. Chen]{Niang Chen}
\address{School of Mathematical Sciences, Laboratory of Mathematics and Complex Systems, Beijing Normal University, Beijing 100875, P.R. CHINA.}
\email{chenniangbnu@mail.bnu.edu.cn}

\author[J.Q. Ge]{Jianquan Ge}
\address{School of Mathematical Sciences, Laboratory of Mathematics and Complex Systems, Beijing Normal University, Beijing 100875, P.R. CHINA.}
\email{jqge@bnu.edu.cn}

\author[M.M. Zhang]{Miaomiao Zhang$^{*}$}
\address{School of Mathematical Sciences, Laboratory of Mathematics and Complex Systems, Beijing Normal University, Beijing 100875, P.R. CHINA.}
\email{miaomiaozhang@mail.bnu.edu.cn}

\date{}
\thanks {$^{*}$ the corresponding author.}
\thanks{The project is partially supported by Beijing Natural Science Foundation (No. Z190003), NSFC (No. 12171037) and the Fundamental Research Funds for the Central Universities.}

\begin{document}
\maketitle

\begin{abstract}
We prove that the index is bounded from below by a linear function of its first Betti number for any compact free boundary $f$-minimal hypersurface in certain positively curved weighted manifolds.
\end{abstract}

\section{Introduction}

Given a Riemannian manifold $(\mathcal{N}^{n+1}, g)$, the free boundary problem consists of finding critical points of the area functional among all compact hypersurfaces $M^n \subset \mathcal{N}$ with $\partial M \subset \partial\mathcal{N}$. Critical points for this problem are minimal hypersurfaces $M \subset \mathcal{N}$ meeting $\partial\mathcal{N}$ orthogonally along $\partial M$, called free boundary minimal hypersurfaces. There are many comparison results between the Morse index and the topology of free boundary minimal hypersurfaces (cf. \cite{Ambrozio, ACS1, AH, BW, BPS, BK, FL}, etc.).

Via embedding the ambient manifold (certain positively curved) in a Euclidean space and using the coordinates of $N\wedge \omega^{\sharp}$ as test functions (where $N$ is a unit normal vector field and $\omega$ is a harmonic 1-form of the hypersurface), Ambrozio,  Carlotto and  Sharp \cite{ACS} proved that the index of any closed minimal hypersurface is bounded from below by a linear function of its first Betti number.
In \cite{ACS1} they used this method to establish such index lower bound for free boundary minimal hypersurfaces in a general manifold with similar extrinsic assumptions (embedding the ambient manifold in a Euclidean space with some integral inequality assumption about curvatures).
In \cite{IRS}, Impera,  Rimoldi and  Savo obtained such index bound for closed and noncompact complete $f$-minimal hypersurfaces in the weighted manifold $(\mathbb{R}^{n+1}, g_{\mathrm{can}}, e^{-f}dV_{\mathbb{R}^{n+1}})$.  Impera and  Rimoldi \cite{IR} obtained the index of closed $f$-minimal hypersurfaces immersed in a general weighted manifold with similar extrinsic assumptions.

In this paper, as a complement to the results above, we consider the index estimates for compact free boundary $f$-minimal hypersurfaces in general weighted manifolds with similar extrinsic assumptions.

\begin{thm}\label{thm-1}
Let $M^{n}$ be a compact, orientable, free boundary $f$-minimal hypersurface of a weighted Riemannian manifold $({\mathcal{N}^{n+1}}, g, e^{-f}dV_\mathcal{N})$. Let $\mathcal{N}^{n+1}$  be isometrically immersed in some Euclidean space $\mathbb{R}^{d}$.
\begin{itemize}
\item[(1)] Assume that for any nonzero tangent $f$-harmonic $1$-form $\omega\in \mathcal{H}_{N f}^{1}(M)$,
\begin{align}
&\int_M\Big(\operatorname{Ric}_{f}^{\mathcal{N}}(\omega^{\sharp}, \omega^{\sharp})+\operatorname{Ric}_{f}^{\mathcal{N}}(N, N)|\omega|^2-K^{\mathcal{N}}(\omega^{\sharp},N)\Big)e^{-f} dV_{M}\nonumber\\
&>\int_M\sum\limits_{k=1}^n\Big(|\mathrm{II}^{\mathcal{N}}(e_k,\omega^{\sharp})|^2+|\mathrm{II}^{\mathcal{N}}(e_k,N)|^2|\omega|^2\Big)e^{-f} dV_{M}\nonumber\\
&\quad-\int_{\partial M}\Big(\mathrm{II}^{\partial \mathcal{N}}(\omega^{\sharp},\omega^{\sharp})+\mathrm{II}^{\partial \mathcal{N}}(N,N)|\omega|^2\Big)e^{-f}dV_{\partial M}.\nonumber
\end{align}
Then $$\operatorname{Index}_{f}(M)\geq \frac{2}{d(d-1)}\mathrm{dim}H^1(M,\mathbb{R}).$$
\item[(2)] Assume that for any nonzero normal $f$-harmonic $1$-form $\omega\in \mathcal{H}_{T f}^{1}(M)$,
\begin{align}
&\int_M\Big(\operatorname{Ric}_{f}^{\mathcal{N}}(\omega^{\sharp}, \omega^{\sharp})+\operatorname{Ric}_{f}^{\mathcal{N}}(N, N)|\omega|^2-K^{\mathcal{N}}(\omega^{\sharp},N)\Big)e^{-f} dV_{M}\nonumber\\
&>\int_M\sum\limits_{k=1}^n\Big(|\mathrm{II}^{\mathcal{N}}(e_k,\omega^{\sharp})|^2+|\mathrm{II}^{\mathcal{N}}(e_k,N)|^2|\omega|^2\Big)e^{-f} dV_{M}\nonumber\\
&\quad-\int_{\partial M}H_f^{\partial\mathcal{N}}|\omega|^2e^{-f}dV_{\partial M}.\nonumber
\end{align}
Then $$\operatorname{Index}_{f}(M)\geq \frac{2}{d(d-1)}\mathrm{dim}H^{n-1}(M,\mathbb{R}).$$
\end{itemize}
Here $\operatorname{Ric}_{f}^{\mathcal{N}}=\operatorname{Ric}^{\mathcal{N}}+\operatorname{Hess}^{\mathcal{N}}f$ denotes the Bakery-Emery Ricci tensor of $\mathcal{N}$; $K^{\mathcal{N}}$ denotes the sectional curvature of $\mathcal{N}$; $\mathrm{II}^{\mathcal{N}}$ denotes the second fundamental form of $\mathcal{N}^{n+1}$ in $\mathbb{R}^{d}$; $\mathrm{II}^{\partial\mathcal{N}}$ denotes the scalar second fundamental form of $\partial\mathcal{N}$ in $\mathcal{N}$ with respect to the inward unit normal vector field $\nu$;  $H_f^{\partial\mathcal{N}}=H^{\partial\mathcal{N}}+\langle\nabla f, \nu\rangle$ is the $f$-mean curvature of $\partial\mathcal{N}$ in $\mathcal{N}$; $H^{\partial\mathcal{N}}=\operatorname{tr}\mathrm{II}^{\partial\mathcal{N}}$ is the mean curvature of $\partial\mathcal{N}$ in $\mathcal{N}$; $N$ is a unit normal vector field of $M$ in $\mathcal{N}$; $\omega^{\sharp}$ is the dual vector field of $\omega$; and $\{e_1,\ldots,e_n\}$ is a local orthonormal frame on $M^{n}$.
\end{thm}

Theorem \ref{thm-1} generalizes all the index estimates of compact hypersurfaces with or without boundary in \cite{ACS, ACS1, IR, IRS}. For example, we apply it to give the following.
\begin{thm}\label{thm-2}
 Let $\mathcal{N}^{n+1}$ be a compact domain in $\mathbb{R}^{n+1}$. Let $M^{n}$ be a free boundary $f$-minimal hypersurface of the weighted manifold $\left({\mathcal{N}^{n+1}}, g_{\mathrm{can}}, e^{-f}dV_\mathcal{N}\right)$ with $\operatorname{Hess}^{\mathcal{N}}f\geq0$.
 \begin{itemize}
\item[(1)] If $\partial\mathcal{N}$ is strictly two-convex in $\mathcal{N}$, then
$$\operatorname{Index}_{f}(M)\geq \frac{2}{n(n+1)}\mathrm{dim}H^{1}(M,\mathbb{R}).$$
\item[(2)] If $\partial\mathcal{N}$ is strictly $f$-mean convex in $\mathcal{N}$, then
$$\operatorname{Index}_{f}(M)\geq \frac{2}{n(n+1)}\mathrm{dim}H^{n-1}(M,\mathbb{R}).$$
\end{itemize}
\end{thm}

When $f=0$, Theorem \ref{thm-2} is a combination of Theorem A and Theorem F of \cite{ACS1}.

\section{Preliminary}
In this section, we prepare the notations and some useful lemmas.

By Nash's embedding theorem, any Riemannian manifold $(\mathcal{N}^{n+1}, g)$ can be isometrically embedded in a sufficiently high-dimensional Euclidean space $\mathbb{R}^{d}$. For an immersed hypersurface $M^n$ of $\mathcal{N}^{n+1}$, let $D$, $\overline{\nabla}$ and $\nabla$ denote the Levi-Civita connection of the Euclidean space, $\mathcal{N}$
and $M$ respectively. The relation between these connections is given by
$$D_{X} Y=\overline{\nabla}_X Y+\mathrm{II}^{\mathcal{N}}(X, Y),$$
$$\overline{\nabla}_X Y=\nabla_{X} Y+\mathrm{II}^{M}(X, Y),$$
where $X,Y$ are vectors fields tangent to $M^n$; $\mathrm{II}^{\mathcal{N}}$ is the second fundamental form of $\mathcal{N}^{n+1}$ in $\mathbb{R}^{d}$; and $\mathrm{II}^{M}$ is the second fundamental form of $M^n$ in $\mathcal{N}^{n+1}$.

A weighted Riemannian manifold $(\mathcal{N}^{n+1}, g,e^{-f}dV_\mathcal{N})$ is a Riemannian manifold endowed with a measure with smooth positive density $e^{-f}$ with respect to the Riemannian volume measure $dV_\mathcal{N}$. We are interested in orientable $f$-minimal hypersurfaces $M$ in $\mathcal{N}$, namely the critical points of the weighted volume functional
$$
\operatorname{Vol}_{f}(M)=\int_{M} e^{-f} d V_{M}.
$$

The first variation formula for a variation $M_{t}$ of $M$ with variation field $\xi$ is given by (see \cite[Lemma 3.2]{CR})
$$
\left.\frac{d}{d t}\operatorname{Vol}_{f}(M_{t})\right|_{t=0}=- \int_{M} u H_{f} e^{-f}d V_{M}+\int_{\partial M} g(\eta, \xi) e^{-f} dV_{\partial M}.
$$
Here $u=g(N,\xi)$; $N$ is the fixed unit normal vector field of $M$ in $\mathcal{N}$; $\eta$ is the outward unit normal vector field along $\partial M$ in $M$; and $H_{f}$ is the $f$-mean curvature of $M$ in $\mathcal{N}$:
$$H_{f} = H^M+\frac{\partial f}{\partial N},$$
where $H^M=\operatorname{tr} g(\mathrm{II}^{M},N)$ is the mean curvature of $M$ in $\mathcal{N}$. Note that $\partial M$ varies inside $\partial \mathcal{N}$ for the free boundary problem. Then $M$ is critical for the $f$-volume, called $f$-minimal with free boundary if and only if $H_{f} = 0$ identically on $M$ and $\eta=-\nu$, where $\nu$ is the inward unit normal vector field of $\partial \mathcal{N}$ in $\mathcal{N}$.

The quadratic form associated to the second variation of the $f$-volume of a free boundary $f$-minimal surface with variation field $\xi=uN$ is (see \cite[Proposition 3.5]{CR})
\begin{align}\label{eq1}
Q_{f}(u, u)=&\int_{M}\left(|\nabla u|^{2}-\left(\operatorname{Ric}_{f}^{\mathcal{N}}(N, N)+|\mathrm{II}^{M}|^{2}\right) u^{2} \right)e^{-f}d V_{M}\\
&-\int_{\partial M}\mathrm{II}^{\partial \mathcal{N}}(N, N) u^{2}e^{-f} d V_{\partial M},\nonumber
\end{align}
where $\operatorname{Ric}_{f}^{\mathcal{N}}=\operatorname{Ric}^{\mathcal{N}}+\operatorname{Hess}^{\mathcal{N}}f$ denotes the Bakery-Emery Ricci tensor of the ambient manifold  $\mathcal{N}$ (see \cite{CR}) and $\mathrm{II}^{\partial \mathcal{N}}(X, Y)=g\left(\overline{\nabla}_{X} Y, \nu\right)$ denotes the scalar second fundamental form of $\partial \mathcal{N}$ in $\mathcal{N}$ with respect to $\nu$.

 The $f$-index of $M$ is the maximal dimension of a linear subspace $V$ in $C^{\infty}(M)$ on which the quadratic form $Q_{f}$ is negative (cf. \cite{IR}). We can also write (\ref{eq1}) making use of the divergence theorem by
$$
\begin{aligned}
Q_{f}(u, u)=&\int_{M}\left( u \Delta_{f} u-\left(\operatorname{Ric}_{f}^{\mathcal{N}}(N, N)+\left|\mathrm{II}^{M}\right|^{2}\right) u^{2}\right) e^{-f}d V_{M}\\
&+\int_{\partial M} u \left(\frac{\partial u}{\partial \eta}-\mathrm{II}^{\partial \mathcal{N}}(N, N)u\right) e^{-f}dV_{\partial M},
\end{aligned}
$$
where $\Delta_{f}u=\Delta u+\langle\nabla f, \nabla u\rangle$, $\Delta u=-\operatorname{div}(\nabla u)$, and the bracket $\langle\cdot,\cdot\rangle$ always denotes the usual inner product between tensors induced by the metric.
The elliptic operator $L_{f}=\Delta_{f}-(\operatorname{Ric}_{f}^{\mathcal{N}}(N, N)+|\mathrm{II}^{M}|^{2})$ is called  the weighted Jacobi operator of $M$.

The boundary condition
$$\frac{\partial u}{\partial \eta}-\mathrm{II}^{\partial \mathcal{N}}(N, N) u=0$$
 makes the weighted Jacobi operator $L_{f}$ self-dual. Under this boundary condition, there exists a non-decreasing and diverging sequence of eigenvalues $\lambda_{1} \leq \lambda_{2} \leq \cdots \leq \lambda_{k} \nearrow \infty$,  associated to a $L^{2}(M, e^{-f} d V_{M})$-orthonormal basis $\left\{u_{k}\right\}_{k=1}^{\infty}$ of solutions to the eigenvalue problem
\begin{equation}\label{LEE}
\begin{cases}
{L}_{f} u=\lambda u, & \text { in } M, \\
\frac{\partial u}{\partial \eta}-\mathrm{II}^{\partial \mathcal{N}}(N, N) u=0, & \text { on } \partial M
.
\end{cases} 
\end{equation}

From the Courant-Hilbert variational characterization for solutions of (\ref{LEE}) (see \cite{CFP, MNG}), if $V_{k}$ denotes the subspace spanned by the first $k$ eigenfunctions for the above problem, then the next eigenvalue $\lambda_{k+1}\left(L_{f}\right)$ equals the minimum of $Q_{f}$ on the $L^{2}(M, e^{-f} d V_{M})$ orthogonal complement of $V_{k}$, i.e.,
$$
\lambda_{k+1}\left(L_{f}\right)=\inf_{u \in V_{k}^\perp \backslash\{0\}} \frac{Q_{f}(u, u)}{\int_{M} u^{2} e^{-f} d V_{M}}.
$$

Since we will use $f$-harmonic 1-forms of the $f$-minimal  hypersurfaces rather than harmonic 1-forms to construct test functions  for the quadratic form $Q_{f}$, we introduce some basic facts about $f$-harmonic 1-forms on manifolds with boundary below.

Let $\iota:\partial M\rightarrow M^n$ and $\iota^*$  denote the natural inclusion and its pull-back map. A $p$-form $\omega\in\Omega^p(M)$ is called $f$-harmonic if $d\omega=0$ and $\delta_f\omega=0$, where $\delta_{f}\omega=\delta \omega +i_{\nabla f}\omega$; $i_{X}$ is the contraction operator from the left hand by $X$; $\delta=(-1)^{n(p+1)+1}\ast d \ast$ is the codifferential operator; and $\ast$  is the Hodge operator with respect to the metric on $M$ (see \cite{IRS}). We denote the spaces of $f$-harmonic $p$-forms tangent and normal at the boundary respectively by
$$
\mathcal{H}_{N f}^{p}(M)=\{\omega \in  \Omega^{p}(M) \mid d \omega=0, \delta_{f} \omega=0~\textit{on} ~M ~\textit{and} ~i_{\eta}\omega=0 ~\textit{on} ~\partial M \},
$$
$$
\mathcal{H}_{Tf}^{p}(M)=\{ \omega \in  \Omega^{p}(M) \mid d\omega=0, \delta_{f}\omega=0~\textit{on} ~M ~\textit{and} ~\iota^*\omega=0 ~\textit{on} ~\partial M \}.
$$
For $f\equiv0$, they are denoted as $\mathcal{H}_{N}^{p}(M)$ and $\mathcal{H}_{T}^{p}(M)$ respectively in \cite{CG}. By the Hodge decomposition, we know  $\mathcal{H}_{N}^{p}(M)\cong H^p(M,\mathbb{R})\cong H_{n-p}(M,\partial M,\mathbb{R})$ and $\mathcal{H}_{N}^{p}(M)\cong\mathcal{H}_{T}^{n-p}(M)$ which still hold in the weighted case (see \cite{Bueler, Yano}). Hence, the dimension of $\mathcal{H}_{N f}^{1}(M)$ equals the first Betti number $b_1(M):=\mathrm{dim}H^1(M,\mathbb{R})$. In fact, the isomorphism $\mathcal{H}_{N}^{1}(M)\cong\mathcal{H}_{N f}^{1}(M)$ (resp. $\mathcal{H}_{T}^{1}(M)\cong\mathcal{H}_{T f}^{1}(M)$) follows directly by setting $\widetilde{\omega}=\omega+du$ for $\omega\in \mathcal{H}_{N}^{1}(M)$ (resp. $\omega\in \mathcal{H}_{T}^{1}(M)$), where $u\in C^{\infty}(M)$ is a solution of the equation
 \begin{equation*}
 \left\{
 \begin{array}{ll}
 \Delta_fu=-i_{\nabla f} \omega & \textit{on~~} M,\\
 \frac{\partial u}{\partial\eta}=0 ~(\textit{resp.~}  \frac{\partial u}{\partial\eta}=-i_{\eta}\omega, u=0) & \textit{on~~} \partial M.
 \end{array}
 \right.
 \end{equation*}
This equation is solvable because
\begin{eqnarray*}
\int_M -i_{\nabla f} \omega e^{-f}dV_M=\int_M-\delta(\omega e^{-f})dV_M=\int_{\partial M} i_\eta(\omega e^{-f})dV_{\partial M}=0.
\end{eqnarray*}
The definition of tangent and normal $f$-harmonic forms for compact manifolds with nonempty boundary also makes the action of the weighted Laplacian operator $\Delta_f=d\delta_f+\delta_fd$ self-dual by the following lemma.
\begin{lem}\label{lem-stokes}
For any two $p$-forms $\omega_1,\omega_2\in\Omega^p(M)$, we have
\begin{equation}\label{eq-stokes}
\begin{array}{ll}
\displaystyle\int_M\Big(\langle \Delta_f \omega_1, \omega_2\rangle-\langle d\omega_1, d\omega_2\rangle-\langle \delta_f \omega_1, \delta_f \omega_2\rangle\Big)e^{-f}dV_M&\\
= -\displaystyle\int_{\partial M}\Big(\langle i_{\eta}d\omega_1,\iota^*\omega_2\rangle-\langle i_{\eta}\omega_2, \iota^*\delta_f \omega_1\rangle\Big)e^{-f}dV_{\partial M}.&
\end{array}
\end{equation}
In particular, a $p$-form $\omega\in\Omega^p(M)$ is tangent $f$-harmonic if and only if $\Delta_f\omega=0$ on $M$, $i_{\eta}d\omega=0$ and $i_{\eta}\omega=0$ on $\partial M$; $\omega\in\Omega^p(M)$ is normal $f$-harmonic if and only if $\Delta_f\omega=0$ on $M$, $\iota^*\omega=0$ and $\iota^*\delta_f\omega=0$ on $\partial M$.
\end{lem}
\begin{proof}
Using the notations of \cite{Yano}, let $\llcorner$ denote the contraction of tensors from the right hand. Then we have
$$\langle\delta\alpha,\beta\rangle-\langle\alpha,d\beta\rangle=\delta(\alpha\llcorner\beta), \quad \textit{for~}\alpha\in\Omega^{p+1}(M), \beta\in\Omega^p(M),$$
and thus
$$\langle\delta_f\alpha,\beta\rangle-\langle\alpha,d\beta\rangle=\langle i_{\nabla f}\alpha,\beta\rangle+\delta(\alpha\llcorner\beta), \quad \textit{for~}\alpha\in\Omega^{p+1}(M), \beta\in\Omega^p(M).$$
It follows that
\begin{eqnarray*}
&\langle\delta_fd\omega_1,\omega_2\rangle e^{-f}-\langle d\omega_1,d\omega_2\rangle e^{-f}=\delta(d\omega_1\llcorner \omega_2 e^{-f}),\\
&\langle d\delta_f\omega_1,\omega_2\rangle e^{-f}-\langle \delta_f\omega_1,\delta_f\omega_2\rangle e^{-f}=-\delta(\omega_2 e^{-f}\llcorner \delta_f\omega_1).
\end{eqnarray*}
Using the Stokes' theorem $$\int_M\delta \theta dV_M=-\int_{\partial M}i_{\eta}\theta dV_{\partial M}, \quad \textit{for~}\theta\in\Omega^1(M),$$
we obtain (\ref{eq-stokes}) by taking sum and integration of the above two formulae.
\end{proof}

In this paper, we use the usual musical isomorphism to pass from 1-forms to vectors, i.e., for a 1-form $\omega\in\Omega^1(M)$, $\omega^{\sharp}$ is the unique vector field on $M$ such that $\omega(Y)=\langle\omega^{\sharp}, Y\rangle$  for all vector fields $Y$.
\begin{lem}\label{lem-nablanorm}
Let $M^n$ be a free boundary $f$-minimal hypersurface of $\mathcal{N}^{n+1}$.
\begin{itemize}
\item[(1)] Let $\omega\in \mathcal{H}_{N f}^{1}(M)$ be a tangent $f$-harmonic $1$-form. Then
\begin{eqnarray*}
\int_M|\nabla\omega|^2e^{-f}dV_M=&-\displaystyle\int_M\left(\operatorname{Ric}_{f}^{\mathcal{N}}(\omega^{\sharp},\omega^{\sharp})-K^{\mathcal{N}}(\omega^{\sharp},N)-|\mathrm{II}^{M}(\cdot,\omega^{\sharp})|^2\right)e^{-f}dV_M\\
&-\displaystyle\int_{\partial M}\mathrm{II}^{\partial \mathcal{N}}(\omega^{\sharp},\omega^{\sharp})e^{-f}dV_{\partial M}.
\end{eqnarray*}
\item[(2)] Let $\omega\in \mathcal{H}_{T f}^{1}(M)$ be a normal $f$-harmonic $1$-form. Then
\begin{eqnarray*}
\displaystyle\int_M|\nabla\omega|^2e^{-f}dV_M=&-\displaystyle\int_M\Big(\operatorname{Ric}_{f}^{\mathcal{N}}(\omega^{\sharp},\omega^{\sharp})-K^{\mathcal{N}}(\omega^{\sharp},N)-|\mathrm{II}^{M}(\cdot,\omega^{\sharp})|^2\Big)e^{-f}dV_M\\
&-\displaystyle\int_{\partial M}(H^{\partial M}+\langle\nabla f, \nu\rangle)|\omega|^2e^{-f}dV_{\partial M}.
\end{eqnarray*}
\end{itemize}
\end{lem}
\begin{proof}
(1)
Since $\omega$  is $f$-harmonic, we have the $f$-Bochner-Weitzenbock formula (cf. \cite{KG})
\begin{align}\label{eq-fBochner}
-\Delta_{f} \frac{|\omega|^{2}}{2}=|\nabla \omega|^{2}+\operatorname{Ric}_{f}^{M}(\omega^{\sharp}, \omega^{\sharp}).
\end{align}
Computing the exterior derivative along $\partial{M}$, we get
$$
d \omega(\eta, \omega^{\sharp})=\langle\nabla_{\eta} \omega^{\sharp}, \omega^{\sharp}\rangle-\langle\nabla_{\omega^{\sharp}} \omega^{\sharp}, \eta\rangle.
$$
Since $d\omega=0$, we have
$$\langle\nabla_{\eta} \omega^{\sharp}, \omega^{\sharp}\rangle=\langle\nabla_{\omega^{\sharp}} \omega^{\sharp}, \eta\rangle.$$
Since $\omega$ is tangent to the boundary, $\omega^{\sharp}$ is a tangent vector field on $\partial M\subset\partial\mathcal{N}$. Then it follows from the free boundary property and Lemma \ref{lem-stokes} that
\begin{eqnarray*}
&\displaystyle\int_{M} \Big(\Delta_{f}\frac{|\omega|^{2}}{2}\Big)e^{-f} dV_{M} =-\displaystyle\int_{\partial M} \Big(i_{\eta}d\frac{|\omega^{\sharp}|^{2}}{2}\Big) e^{-f}dV_{\partial M}=- \displaystyle\int_{\partial M}\langle\nabla_{\eta}\omega^{\sharp}, \omega^{\sharp}\rangle e^{-f}dV_{\partial M} \\
&=- \displaystyle\int_{\partial M}\langle\nabla_{\omega^{\sharp}} \omega^{\sharp}, \eta\rangle e^{-f}dV_{\partial M}
= \displaystyle\int_{\partial M}\langle\overline{\nabla}_{\omega^{\sharp}} \omega^{\sharp}, \nu\rangle e^{-f}dV_{\partial M}
=\displaystyle \int_{\partial M}\mathrm{II}^{\partial \mathcal{N}}(\omega^{\sharp}, \omega^{\sharp})e^{-f}d   V_{\partial M}.
\end{eqnarray*}
Integrating (\ref{eq-fBochner}) we get
\begin{equation}\label{eq-int-nabla}
\int_M\Big(|\nabla \omega|^{2}+\operatorname{Ric}_{f}^{M}(\omega^{\sharp}, \omega^{\sharp})\Big)e^{-f}dV_M= -\int_{\partial M}\mathrm{II}^{\partial \mathcal{N}}(\omega^{\sharp}, \omega^{\sharp})e^{-f}d   V_{\partial M}.
\end{equation}
The Gauss equation for $f$-minimal hypersurfaces implies
\begin{equation}\label{eq-gauss}
\operatorname{Ric}_{f}^{M}(\omega^{\sharp}, \omega^{\sharp})=\operatorname{Ric}_{f}^{\mathcal{N}}(\omega^{\sharp},\omega^{\sharp})-K^{\mathcal{N}}(\omega^{\sharp},N)-|\mathrm{II}^{M}(\cdot,\omega^{\sharp})|^2.
\end{equation}
Substituting (\ref{eq-gauss}) in (\ref{eq-int-nabla}) we obtain the required formula.

(2) Let $\{e_k\}_{k=1}^{n-1}$ be a local orthonormal frame of $\partial M$. Since $\omega$  is normal $f$-harmonic, $\delta_f\omega=0$ and $\omega^\sharp=\lambda\eta$ for some function $\lambda$ on $\partial M$. Then on $\partial M$,
$$0=\delta_f\omega=-\sum_{k=1}^{n-1}\langle\nabla_{e_k}\omega^{\sharp}, e_k\rangle-\langle\nabla_{\eta}\omega^{\sharp}, \eta\rangle+\langle\nabla f, \omega^\sharp\rangle=-\lambda H^{\partial M}-\langle\nabla_{\eta}\omega^{\sharp}, \eta\rangle+\lambda\langle\nabla f, \eta\rangle,$$
where $H^{\partial M}=\operatorname{tr} \langle\mathrm{II}^{\partial M},-\eta\rangle$ is the mean curvature of $\partial M$ in $M$.

Thus, since $\eta=-\nu$ on $\partial M$, we have
$$\int_{M} \Big(\Delta_{f}\frac{|\omega|^{2}}{2}\Big)e^{-f} dV_{M} =-\int_{\partial M}\langle\nabla_{\eta}\omega^{\sharp}, \omega^{\sharp}\rangle e^{-f}dV_{\partial M}=\int_{\partial M}(H^{\partial M}+\langle\nabla f, \nu\rangle)|\omega|^2e^{-f}dV_{\partial M}.$$
Then the required formula follows from (\ref{eq-fBochner}) and (\ref{eq-gauss}).
\end{proof}

To estimate the index, we use the coordinates of $N\wedge \omega^{\sharp}$ for $f$-harmonic 1-forms $\omega$ as the test functions.

\begin{lem}\label{lem-testindex}
Let $M^{n}$ be a free boundary $f$-minimal hypersurface of a weighted manifold $\left({\mathcal{N}^{n+1}}, g, e^{-f}dV_\mathcal{N}\right)$. Let $\mathcal{N}^{n+1}$  be isometrically immersed in some  Euclidean space $\mathbb{R}^{d}$. For a $f$-harmonic $1$-form $\omega$ of $M$, let
$$
u_{ij}=\langle N\wedge \omega^{\sharp}, \theta_{i}\wedge\theta_j\rangle, \quad 1\leq i<j\leq d,
$$
be the coordinates of $N\wedge \omega^{\sharp}$ under a fixed orthonormal basis $\left\{\theta_{i}\wedge\theta_j\mid 1\leq i<j\leq d\right\}$ of $\Omega^2(\mathbb{R}^{d})$.
Let $\{e_1,\cdots,e_n\}$ be a local orthonormal frame of $M^n$.
\begin{itemize}
\item[(1)]  Let $\omega\in \mathcal{H}_{N f}^{1}(M)$ be a tangent $f$-harmonic $1$-form. Then
\begin{align}
\sum\limits_{1\leq i<j\leq d}Q_f(u_{ij},u_{ij})=&\int_M\sum\limits_{k=1}^n\Big(|\mathrm{II}^{\mathcal{N}}(e_k,\omega^{\sharp})|^2+|\mathrm{II}^{\mathcal{N}}(e_k,N)|^2|\omega|^2\Big)e^{-f} dV_{M}\nonumber\\
&-\int_M\Big(\operatorname{Ric}_{f}^{\mathcal{N}}(\omega^{\sharp}, \omega^{\sharp})+\operatorname{Ric}_{f}^{\mathcal{N}}(N, N)|\omega|^2-K^{\mathcal{N}}(\omega^{\sharp},N)\Big)e^{-f} dV_{M}\nonumber\\
&-\int_{\partial M}\Big(\mathrm{II}^{\partial \mathcal{N}}(\omega^{\sharp},\omega^{\sharp})+\mathrm{II}^{\partial \mathcal{N}}(N,N)|\omega|^2\Big)e^{-f}dV_{\partial M}.\nonumber
\end{align}
\item[(2)] Let $\omega\in \mathcal{H}_{T f}^{1}(M)$ be a normal $f$-harmonic $1$-form. Then
\begin{align}
\sum\limits_{1\leq i<j\leq d}Q_f(u_{ij},u_{ij})=&\int_M\sum\limits_{k=1}^n\Big(|\mathrm{II}^{\mathcal{N}}(e_k,\omega^{\sharp})|^2+|\mathrm{II}^{\mathcal{N}}(e_k,N)|^2|\omega|^2\Big)e^{-f} dV_{M}\nonumber\\
&-\int_M\Big(\operatorname{Ric}_{f}^{\mathcal{N}}(\omega^{\sharp}, \omega^{\sharp})+\operatorname{Ric}_{f}^{\mathcal{N}}(N, N)|\omega|^2-K^{\mathcal{N}}(\omega^{\sharp},N)\Big)e^{-f} dV_{M}\nonumber\\
&-\int_{\partial M}(H^{\partial\mathcal{N}}+\langle\nabla f, \nu\rangle)|\omega|^2e^{-f}dV_{\partial M}.\nonumber
\end{align}
\end{itemize}
\end{lem}

\begin{proof}
As $N$ is a unit normal vector field of $M^n$ in $\mathcal{N}^{n+1}$ and $\omega^{\sharp}$ is a tangent vector field of $M^n$, we have $$|N\wedge \omega^{\sharp}|^2=\sum\limits_{i<j}u_{ij}^2=|\omega|^2.$$
Substituting the test functions $u_{ij}$ in (\ref{eq1}), we get
\begin{align}\label{index-testsum}
\sum\limits_{i<j}Q_f(u_{ij},u_{ij})=&\int_{M}\Big(\sum\limits_{i<j}|\nabla u_{ij}|^{2}-(\operatorname{Ric}_{f}^{\mathcal{N}}(N, N)+|\mathrm{II}^{M}|^{2}) |\omega|^2 \Big)e^{-f}d V_{M} \\
&-\int_{\partial M}\mathrm{II}^{\partial \mathcal{N}}(N, N) |\omega|^2e^{-f} d V_{\partial M}.\nonumber
\end{align}

Note that $$\sum\limits_{i<j}|\nabla u_{ij}|^2=\sum\limits_{i<j}\sum_{k=1}^n\langle D_{e_k}(N\wedge \omega^{\sharp}),\theta_i\wedge\theta_j\rangle^2=\sum_{k=1}^n |D_{e_k}(N\wedge \omega^{\sharp})|^2.$$
Using the orthogonal decompositions, we compute the squared norms as follows:
\begin{align}
&|D_{e_k}(N\wedge \omega^{\sharp})|^2=|D_{e_k}N\wedge \omega^{\sharp}+N\wedge D_{e_k}\omega^{\sharp}|^2 \nonumber\\
&=\left|(\overline{\nabla}_{e_k}N+\mathrm{II}^{\mathcal{N}}(e_k,N))\wedge\omega^{\sharp}+N\wedge(\overline{\nabla}_{e_k}\omega^{\sharp}+\mathrm{II}^{\mathcal{N}}(e_k,\omega^{\sharp}))\right|^2\nonumber\\
&=\left|(-Ae_k+\mathrm{II}^{\mathcal{N}}(e_k,N))\wedge\omega^{\sharp}+N\wedge(\nabla_{e_k}\omega^{\sharp}+\mathrm{II}^{\mathcal{N}}(e_k,\omega^{\sharp}))\right|^2\nonumber\\
&=|Ae_k\wedge\omega^{\sharp}|^2+|\mathrm{II}^{\mathcal{N}}(e_k,N)\wedge\omega^{\sharp}|^2+|N\wedge\nabla_{e_k}\omega^{\sharp}|^2+|N\wedge\mathrm{II}^{\mathcal{N}}(e_k,\omega^{\sharp})|^2\nonumber\\
&=|Ae_k|^2|\omega|^2-\langle Ae_k,\omega^{\sharp}\rangle^2+|\mathrm{II}^{\mathcal{N}}(e_k,N)|^2|\omega|^2+|\nabla_{e_k}\omega|^2+|\mathrm{II}^{\mathcal{N}}(e_k,\omega^{\sharp})|^2,\nonumber
\end{align}
where $A$ is the shape operator of $M^n$ such that $\langle AX,Y\rangle=\langle \mathrm{II}^{M}(X,Y), N\rangle$. Thus
\begin{equation*}
\sum\limits_{i<j}|\nabla u_{ij}|^2=|\mathrm{II}^{M}|^2|\omega|^2-| \mathrm{II}^{M}(\cdot,\omega^{\sharp})|^2+\sum_{k=1}^n\Big(|\mathrm{II}^{\mathcal{N}}(e_k,N)|^2|\omega|^2+|\mathrm{II}^{\mathcal{N}}(e_k,\omega^{\sharp})|^2\Big)+|\nabla\omega|^2.
\end{equation*}
By the free boundary property, i.e., $\eta=-\nu$ along $\partial M\subset\partial\mathcal{N}$, $N$ is also a unit normal vector field of $\partial M$ in $\partial\mathcal{N}$. It follows that
$$H^{\partial M}+\mathrm{II}^{\partial \mathcal{N}}(N, N)=H^{\partial\mathcal{N}}.$$
Putting these in (\ref{index-testsum}) and applying Lemma \ref{lem-nablanorm}, we obtain both formulae for the two cases of the Lemma.
\end{proof}

\section{Proof of the main theorems}\label{MTP}
We are now ready to prove the main theorems.

\textbf{Proof of Theorem \ref{thm-1}.}~ Let $k$ be the $f$-index of $M$, that is the number of negative eigenvalues of the weighted Jacobi operator $L_f$ of $M$ in (\ref{LEE}), and denote by $\{\phi_{p}\}_{p=1}^{\infty}$ a $L^{2}(M, e^{-f} d V_{M})$-orthonormal basis of the eigenfunctions  corresponding to the eigenvalues $\lambda_{1} \leq \lambda_{2} \leq  \ldots \leq \lambda_{k} < 0\leq \lambda_{k+1}\cdots$ of $L_f$. Following the approach of Ambrozio-Carlotto-Sharp \cite{ACS, ACS1}, we define the following linear map
$$
\begin{aligned}
\Phi: \mathcal{H}_{f}^{1}(M) & \longrightarrow \mathbb{R}^{d(d-1)k/2} \\
\omega & \longmapsto\left[\int_{M} u_{ij} \phi_{p} e^{-f}d V_{M}\right],
\end{aligned}
$$
where $\mathcal{H}_{f}^{1}(M)=\mathcal{H}_{Nf}^{1}(M)\cong H^1(M,\mathbb{R})$ in case $(1)$ and $\mathcal{H}_{f}^{1}(M)=\mathcal{H}_{Tf}^{1}(M)\cong H^{n-1}(M,\mathbb{R})$ in case $(2)$ of Theorem \ref{thm-1} respectively, $u_{ij}=\langle N\wedge \omega^{\sharp}, \theta_{i}\wedge\theta_j\rangle$ are the test functions as in Lemma \ref{lem-testindex},  $1\leq i<j\leq d$, $1\leq p\leq k$.

Assume by contradiction that $\mathrm{dim}\mathcal{H}_{f}^{1}(M)>d(d-1)k/2$. Then there would exist a nonzero $f$-harmonic $1$-form $\omega\in\mathcal{H}_{f}^{1}(M)$ such that
$\int_{M} u_{ij} \phi_{p} e^{-f}d V_{M}=0$ for all $1\leq i<j\leq d$ and all $p=1,\ldots,k$. This means that each $u_{ij}$ is $L^{2}(M, e^{-f} d V_{M})$-orthogonal to all the first $k$ eigenfunctions $\phi_{p}$. Thus from the Courant-Hilbert variational characterization of eigenvalues it follows that
$$
\sum_{i<j} Q_{f}(u_{ij}, u_{ij})\geq  \lambda_{k+1}\sum_{i<j}  \int_{M}u_{ij}^{2}e^{-f} d V_{M}= \lambda_{k+1}\int_{M}|\omega|^2e^{-f} d V_{M}\geq 0.
$$
In view of Lemma \ref{lem-testindex},  this is a contradiction with either hypothesis of Theorem \ref{thm-1}. \qed

\textbf{Proof of Theorem \ref{thm-2}.}~ Now since $\mathcal{N}^{n+1}$ is totally geodesic in $\mathbb{R}^{n+1}$, $\operatorname{Ric}_{f}^{\mathcal{N}}=\operatorname{Hess}^{\mathcal{N}}f\geq0$, $K^{\mathcal{N}}=0$ and $\mathrm{II}^{\mathcal{N}}=0$. Hence, if $\partial\mathcal{N}$ is strictly two-convex in $\mathcal{N}$, or if $\partial\mathcal{N}$ is strictly $f$-mean convex in $\mathcal{N}$, then the corresponding inequality assumption of Theorem \ref{thm-1} is satisfied and thus the conclusion follows. \qed


\begin{thebibliography}{99}

\bibitem{Ambrozio}
L. Ambrozio, \emph{Rigidity of area-minimizing free boundary surfaces in three-manifolds}, J. Geom. Anal. \textbf{25} (2015), 1001--1017.

\bibitem{ACS}
L. Ambrozio, A. Carlotto and B. Sharp, \emph{Comparing the Morse index and the first Betti number of minimal hypersurfaces}, J. Differ. Geom. \textbf{108} (2018), 379--410.

\bibitem{ACS1}
L. Ambrozio, A. Carlotto and B. Sharp,\emph{Index estimates for free boundary minimal hypersurfaces}, Math. Ann. \textbf{370} (2018), 1063--1078.

\bibitem{AH}
N. S. Aiex and H. Han, \emph{Index estimates for surfaces with constant mean curvature in 3-dimensional manifolds}, Calc. Var.  Partial Differential Equations. \textbf{60} (2021), 3--20.

\bibitem{BW}
E. Barbosa and Y. Wei, \emph{A compactness theorem of the space of free boundary $f$-minimal surfaces in three-dimensional smooth metric measure space with boundary}, J. Geom. Anal. \textbf{26} (2016), 1995-2012.

\bibitem{BPS}
 R. G. Bettiol, P. Piccione and B. Santoro, \emph{Deformations of free boundary CMC hypersurfaces}, J. Geom. Anal. \textbf{27} (2017), 3254--3284.

\bibitem{Bueler}
E. L. Bueler, \emph{The heat kernel weighted Hodge Laplacian on noncompact manifolds}, Trans. Amer. Math.Soc. \textbf{351} (1999), 683--713.

\bibitem{BK}
W. B\"{u}rger and E. Kuwert, \emph{Area-minimizing disks with free boundary and prescribed enclosed volume}, J. Reine Angew. Math. \textbf{621} (2008), 1--27.

\bibitem{CFP}
 J. Chen, A. Fraser and C. Pang, \emph{Minimal immersions of compact bordered Riemann surfaces with free boundary}, Trans. Amer. Math. Soc. \textbf{367} (2015), 2487--2507.

\bibitem{CG}
N. Chen and J. Q. Ge, \emph{Cohomology vanishing theorems for free boundary submanifolds}, arxiv:2106.05793.

\bibitem{CR}
K. Castro, C. Rosales, \emph{Free boundary stable hypersurfaces in manifolds with density and rigidity results}, J. Geom. Phys. \textbf{79} (2014), 14--28.

\bibitem{FL}
A. Fraser and M. Li, \emph{Compactness of the space of embedded minimal surfaces with free boundary in three-manifolds with nonnegative Ricci curvature and convex boundary}, J. Differ. Geom. \textbf{96} (2014), 183--200.

\bibitem{IR}
D. Impera and  M. Rimoldi, \emph{Index and first Betti number of $f$-minimal hypersurfaces: general ambients}, Ann. Mat. Pura Appl. \textbf{199} (2020), 2151--2165.

\bibitem{IRS}
D. Impera, M. Rimoldi and A. Savo, \emph{Index and first Betti number of $f$-minimal hypersurfaces and self-shrinkers}, Rev. Mat. Iberoam. \textbf{36} (2020), 817--840.

\bibitem{KG}
S. Keomkyo and Y. Gabjin, \emph{ Liouville-type theorems for weighted $p$-harmonic 1-forms and weighted $p$-harmonic maps}, Pacific J. Math. \textbf{305} (2020),  291--310.

\bibitem{MNG}
D. Maximo, I. Nunes and G. Smith, \emph{Free boundary minimal annuli in convex three-manifolds}, J. Differ.Geom. \textbf{106} (2017), 139--186.

 \bibitem{Yano}
K. Yano, \emph{Integral formulas in Riemannian geometry}, Pure and Applied Mathematics, vol. \textbf{1} Marcel Dekker, Inc., New York (1970).



\end{thebibliography}
\end{document}